\documentclass[12pt]{article}

\usepackage{amsmath,amsfonts,amssymb,amsthm,amscd,color}

\title{Some model theory and topological dynamics of $p$-adic algebraic groups}
\date{\today}

\author{Davide Penazzi \thanks{Partially supported by IMA Small Grant SGS21/16 and NSF grant DMS-141947}\\University of Central Lancashire\and  Anand Pillay \thanks{Partially supported by NSF grants DMS-1360702, DMS-1665035, and DMS-1760413}\\University of Notre Dame\and  Ningyuan Yao\thanks{Partially supported by NSFC grant 11601090, Initial Scientific Research Fund of Young Teachers in Fudan University, and William J. Hank Family Chair research funds (Notre Dame) }\\Fudan University}

\newtheorem{Theorem}{Theorem}[section]
\newtheorem{Prop}[Theorem]{Proposition}
\newtheorem{Def}[Theorem]{Definition}
\newtheorem{Rmk}[Theorem]{Remark}
\newtheorem{Lemma}[Theorem]{Lemma}
\newtheorem{Cor}[Theorem]{Corollary}
\newtheorem{Fact}[Theorem]{Fact}

\newtheorem{Question}[Theorem]{Question}

\newcommand{\R}{\mathbb R}
\newcommand{\Q}{{\mathbb Q}_p}
\newcommand{\Z}{{\mathbb Z}_p}
\newcommand{\z}{\mathbb Z}
\newcommand{\N}{\mathbb N}

\newcommand{\pp}{\mathbb P}

\newcommand{\K}{\mathbb K}
\newcommand{\V}{\mathbb O}
\newcommand{\Ga}{{\mathbb G}_{a}}
\newcommand{\Gm}{{\mathbb G}_{m}}

\DeclareMathOperator{\tp}{tp}

\begin{document}

\maketitle
\begin{abstract}
We initiate the study of $p$-adic algebraic groups $G$  from the stability-theoretic and 
definable topological-dynamical points of view, that is,  we consider invariants of the action of $G$ on its space of types over $\Q$ in the language of fields.   We consider the additive and multiplicative groups of $\Q$ and $\Z$, the group of upper triangular invertible $2\times 2$ matrices, $SL(2,\Z)$, and, our main focus, $SL(2,\Q)$.  In all cases we identify $f$-generic types (when they exist),  minimal subflows, and idempotents.  Among the main results is that the ``Ellis group" of $SL(2,\Q)$ is ${\hat \z}$, yielding a counterexample to Newelski's conjecture with new features:  $G = G^{00} =  G^{000}$ but the Ellis group is infinite.  A final section deals with the action of $SL(2,\Q)$ on the type-space of the projective line over $\Q$.

\end{abstract}

\bibliographystyle{plain}
\section{Introduction and preliminaries}
The machinery of topological dynamics has proved to be useful in generalizing stable group theory to unstable environments (the original paper on the topic being \cite{Newelski}). Given a structure $M$ and group $G$ definable in $M$, a natural action, given by model theory,  is that of $G$ on the space $S_{G}(M)$ of complete types over $M$ concentrating on $G$.
On the other hand this action is simply a dynamical system for  $G$  considered as a discrete group.  When $Th(M)$ is stable,  $G$ is what is called a stable group, and the fundamental theorems of stable group theory are coded in this dynamical system.   There has been a considerable amount of work extending stable group theory to the case where $Th(M)$ is $NIP$ (does not have the independence  property), and $G$ is definably amenable  (see \cite{NIPII} and \cite{C-S} for example).  When  $M$ is the field of reals, then a (semialgebraic) noncompact simple Lie group such as $SL(2,\R)$ is {\em not} definably amenable, but its definable topological dynamics was nevertheless analyzed in \cite{GPPI}. The latter work was partly motivated by a conjecture of Newelski on the connection between the Ellis group of such an action and the definable Bohr compactification $G^{*}/(G^{*})^{00}$ of $G$ ($G^{*}$ being the interpretation of $G$ in a saturated elementary extension). The case of $SL_(2,\R)$ gives an example where these two invariants are different, the definable Bohr compactification being trivial and the Ellis group being $\z/2\z$.  In the current paper we extend this analysis of \cite{GPPI} to  the $p$-adic context,  namely where $M$ is the field of $p$-adic numbers, rather than the field of reals. We focus on $SL(2,\Q)$ and its building blocks  but the analysis should extend to semisimple $p$-adic Lie groups (as groups definable in the $p$-adic field).   

In the real case we made use of the Iwasawa decomposition of $SL(2,\R)$ as 
\newline
$B(\R)^{0}\cdot SO(2,\R)$, where $B$ is the Borel subgroup of upper triangular matrices,  $SO(2,\R)$ is a maximal compact subgroup, and $B(\R)^{0}$ is the semialgebraic, equivalently topological, connected component of $B(\R)$ (note that the intersection of $B(\R)^{0}$ and $SO(2,\R)$ is trivial).  In the $p$-adic case, the Iwasawa decomposition of $SL(2,\Q)$ has the form 
\newline
$B(\Q)\cdot SL(2,\Z)$ where $B$ is as before, and now $SL(2,\Z)$ is a maximal compact subgroup. However now the intersection of the constituents is large (in fact of $p$-adic dimension $2$) and also the constituents are far from connected. For example $SL(2,\Z)$, being profinite, has trivial connected component. So  the analysis in the $p$-adic case is rather harder and requires some new ideas.  A crucial role in our analysis  of $SL(2,\R)$ was its action on the homogeneous space $SL(2,\R)/B(\R)^{0}$, which is a $2$-cover of the natural action of $SL(2,\R)$ on $\pp^{1}(\R)$. In fact the universal minimal definable flow of $SL(2,\R)$ was the space of nonalgebraic types of the homogeneous space $SL(2,\R)/B(\R)^{0}$. We proceed quite differently in the $p$-adic case. On the other hand,  there are analogies between the final statements regarding the Ellis group; in the real case the Ellis group of $SL(2,\R)$ (acting on its type space) is $\z/2\z$ which coincides with $\K^{*}/(\K^{*})^{0}$ where $\K^{*}$ is the muultiplicative group of a saturated real closed field $\K$. In the $p$-adic case, the Ellis of group of $SL(2,\Q)$ (acting on its type space) is $\hat\z$ which coincides with  $\K^{*}/(\K^{*})^{0}$ where $\K^{*}$ is the multiplicative group a saturated $p$-adically closed field $\K$.  In any case, $SL(2,\Q)$ provides another counterexample to Newelski's conjecture on the relationship between the Ellis group and $G/G^{00}$, but with different features from the ones provided by Corollary 0.3 of \cite{Krupinski-PillayII} for example, as the Ellis group is infinite whereas $G = G^{00} = G^{000}$.

To be more precise, our main results are as follows where $M$ denotes the structure $(\Q,+,\times)$, $G$ denotes $SL(2,-)$, $S_{G}(M)$ denotes the space of complete types over $M$ extending the formula `$x\in G$',  $*$ denotes the canonical semigroup structure on $S_{G}(M)$, and other notation will be explained later.

\begin{itemize}
\item   A minimal subflow of $(G(M), S_{G}(M))$ is $cl({\cal I}*{\cal J})$ where ${\cal I}$ is the unique minimal subflow of the action of $SL(2,\Z)$ on its type space, and ${\cal J}$ is a certain minimal subflow of the action of $B(\Q)$ on its type space. In particular $cl({\cal I}*{\cal J})$ is the universal minimal definable flow of $SL(2,\Q)$. See Theorem  3.4. 

\item The Ellis group attached to the flow $(G(M), S_{G}(M))$ is $\hat\z$.   See Corollary 3.8. 

\end{itemize}

We also prove that the space of nonalgebraic types over $M$ of the projective line $\\P^{1}(\Q)$ is minimal and proximal under the natural action of $SL(2,\Q)$. See Corollary 4.8. 

As part of our analysis we classify  $1$-types over $\Q$ from the stable group theory  point of view, namely we describe $f$-generics of various kinds (definable, finitely satisfiable)  and minimal flows, with respect to the additive and multiplicative groups. This does not seem to have observed  before, and provides  interesting phenomena for definable topological dynamics in the $NIP$ setting. 

Let us discuss where our  work fits into current themes in topological dynamics and definable groups. 
This paper does not explicitly  offer any new general results in  topological dynamics and model theory.  However, the project of generalizing  the study of groups definable in $o$-minimal structures to the $p$-adic environment has been on the cards for a long time.  Benjamin Druart's thesis \cite{Druart-thesis} and the preprint  \cite{Druart-paper}, studied groups definable in $p$-adically closed fields, in particular $SL(2,\Q)$, in analogy with $o$-minimal and finite Morley rank groups  methods. On the other hand there has been considerable interest in generalising the stability-theoretic and topological dynamical study of real Lie groups such as $SL(2,\R)$ to the $p$-adic context since the paper \cite{GPPI} was written in 2012, and this is what we accomplish in the current paper.  Moreover, as G. Jagiella has pointed out to us, the methods in our paper suggest generalizations to definable groups $G$  in $NIP$ theories with a decomposition $G = B\cdot K$ where $B$ has ``definable $f$-generics" and $K$ has ``finitely satisfiable generics".  This will be pursued in  future work.

Our notation for model theory is standard, and we will assume familiarity with basic notions such as type spaces, heirs, coheirs, definable types etc.  References are \cite{Poizat-book} as well as \cite{Pillay-intro-stability}.

Our notation for the $p$-adics is as follows:  $\Q$ is the field of $p$-adics and  $\Z$ is the ring of $p$-adic integers. $\z$ is the ordered additive group of integers, the value group of $\Q$. 
$M$ denotes the standard model $(\Q,+,\times, -, 0, 1)$, and we  sometimes write $\Q$ for $M$. $\bar M$ denotes a saturated elementary extension  $(\K, +, \times,0,1)$ of $M$ and again sometimes we write $\K$ for $\bar M$.  $\Gamma$ denotes the value group of $\K$.  We will be  referring a lot  to the comprehensive survey \cite{Belair},  for the basic model theory of the $p$-adics.  A key point is Macintyre's theorem \cite{Macintyre} that $Th(\Q,+,\times,0,1)$ has quantifier elimination in the language where we add predicates $P_{n}(x)$ for the $nth$ powers (all $n$). Moreover the valuation is quantifier-free definable in Macintyre's language, in particular is definable in the language of rings.  (See Section 3.2 of \cite{Belair}.) We will give a little more background at the beginning of Section 2.1.

In the rest of this introduction we give more background on topological dynamics and the model-theoretic approach. 

In Section 2 we analyse the model-theoretic dynamics of the  building blocks of $SL(2,\Q)$, namely the additive and multiplicative groups, the Borel subgroup, and the maximal compact subgroup $SL(2,\Z)$.  As mentioned earlier, this is of independent interest. 

In Section 3, we prove the main results, on the minimal subflows and Ellis group of the action on $SL(2,\Q)$ on its type space, making use of the Iwasawa decomposition and results in Section 2. 

In  Section 4 we study the action  on the type space of the projective line. We also ask several questions. 

\vspace{2mm}
\noindent
We would like to thank a referee for his/her comments on a first version of this paper.  He/she pointed out many  mathematical points which needed clarification and/or correction, often suggesting the required correction. Following these comments, we have also added explanations of how the current  paper differs from the earlier work (\cite{GPPI}) on $SL(2,\R)$, and how it relates to other current research.

\subsection{Topological dynamics}

Our references for (abstract) topological dynamics are \cite{Auslander} and \cite{Glasner}

Given a (Hausdorff) topological group $G$, by a $G$-flow mean a continuous action $G\times X\to X$ of $G$ on a compact (Hausdorff) topological space $X$. We sometimes write the flow as $(X,G)$. Often it is assumed that there is a dense orbit, and sometimes a $G$-flow $(X,G)$ with a distinguished point $x\in X$ whose orbit is dense is called a $G$-ambit.

In spite of $p$-adic algebraic groups being nondiscrete topological groups, we will be treating them as discrete groups so as to have their actions on type spaces being contiinuous.
(But note that there is a model-theoretic account of the dynamics of definable groups with a definable topology. See \cite{Krupinski-Pillay} for example. And it might be worthwhile to prove and compare results.

So in this background section we assume $G$ to be a discrete group, in which case a $G$-flow is simply an action of $G$ by homeomorphisms on a compact space $X$. 

By a subflow of $(X,G)$ we mean a closed $G$-invariant subspace $Y$ of $X$ (together with the action of $G$ on $Y$). $(X,G)$ will always have  minimal nonempty subflows. 
Points $x,y\in X$ are {\em proximal} with respect to $(X,G)$ if there is a net $(g_{\alpha})_{\alpha}$ in $G$ and $z$ in $X$ such that  both $g_{\alpha}x$ and $g_{\alpha}y$ converge to $z$.
$(X,G)$ is proximal if every pair of elements of $X$ is proximal. 

Given a flow $(X,G)$,  its enveloping semigroup $E(X)$  is the closure in the space $X^{X}$ (with the product topology) of the set of maps $\pi_{g}:X\to X$, where $\pi_{g}(x) = gx$, equipped with composition $\circ$ (which is continuous on the left).   So any $e\in E(X)$ is a map from $X$ to $X$ and, for  example, proximality of the flow $(X,G)$ is equivalent to: for all $x,y\in X$ there is  $e\in E(X)$ such that $e(x) = e(y)$. 

Note also that $E(X)$ is a compact space and we have an action $g\cdot e = \pi_{g}\circ e$ of $G$ on $E(X)$, by homeomorphisms.  So $(E(X), G)$ is a flow too.

Ellis \cite{Ellis} proved the following correspondence between minimal subflows and ideals of $E(X)$.

\begin{Theorem}\label{ellisthm}
Denote by $J$ the set of idempotents of the enveloping semigroup $E(X)$. Then
\begin{enumerate}
 \item Minimal (automatically closed) left ideals $I$ of $E(X)$ coincide with minimal subflows.
 \item Given a minimal closed left ideal $I$, $I\cap J\neq \emptyset$; moreover for $u\in I\cap J$, $(u\circ I,\circ)$ is a group, called the Ellis group.
 \item All Ellis groups (varying $I$ and $u$) are isomorphic, so we sometimes refer to this isomorphism class as the Ellis group  attached to the original flow $(X,G)$.
\end{enumerate}
\end{Theorem}

There is a {\em universal} $G$-ambit, which is (under our discreteness assumption on $G$) the Stone-Cech compactification $\beta G$ of $G$. This is precisely the (Stone) space of ultrafilters on the Boolean algebra of {\em all} subsets of $G$. The action of $G$ on itself by left translation gives rise to an action on $\beta G$ by homeomorphisms. Identifying $g\in G$ with the principal ultrafilter it generates yields an embedding of $G$ in $\beta G$ and $\beta G$ together with the identity element $id_{G}$  of $G$ as distinguished point, is the universal $G$-ambit. The universal property is that for any other $G$-ambit $(X,x)$ there is a unique map of $G$
flows from $\beta G$ to $X$ which takes $id_{G}$ to $X$. 

The enveloping semigroup $E(\beta G)$ of $\beta G$ coincides with $\beta G$ (due to its universal character) and  hence $\beta G$ is equipped with a canonical semigroup structure, continuous on the left. It can be described explicitly in various ways, see Section 4 of \cite{Newelski} for one such description.   Minimal $G$-subflows, equivalently minimal left ideals, of $\beta G$ are isomorphic as $G$-flows and coincide with the {\em universal minimal  $G$-flow} $({\cal M},G)$, a $G$-flow, unique up to isomomorphism, with the feature that any minimal $G$-flow is an (surjective)  image of $({\cal M},G)$ under a map of $G$-flows.   The Ellis group of $(\beta G, G)$ is an important invariant of the group $G$. 

\subsection{Model theory} 
The model-theoretic background for the current paper is contained in \cite{GPPII} and \cite{P-Y}, but we give a quick summary here.

Given an $L$-structure $M$, $S(M)$ denotes the collection of all complete types over $M$ (in all sorts or number of variables). For a definable set $Z$ in $M$, $S_{Z}(M)$ denotes the (Stone) space of complete types over $M$ containing the formula $x\in Z$.  If $M'$ is an elementary extension of $M$, $Z(M')$  denotes the interpretation in $M'$ of the formula defining $Z$ in $M$. $\bar M$ denotes a saturated elementary extension of $M$, and we also may consider elementary extensions of $\bar M$ in which all types over $\bar M$ are realized.

\begin{Fact} Suppose that all complete types (in any sort) over $M$ are definable. Then every complete type $p$ over $M$ has a unique coheir in $S(\bar M)$ as well as a unique heir in $S(\bar M)$.
\end{Fact}

This applies to the situation where $M = (\Q,+,\times, -, 0, 1)$ due to a theorem of Delon \cite{Del89}. 

Now suppose that $G$ is a group definable in $M$. 

The ``definable"  analogue of $\beta G$ is the space $S_{G}(M)$ of all complete types over $M$ concentrating on $G$. $G$ clearly acts on $S_{G}(M)$ (on the left)  by homeomorphisms.  If all types over $M$ are definable, then $E(S_{G}(M))$ coincides with $S_{G}(M)$ and we already have a semigroup structure, which we denote by $*$,  on $S_{G}(M)$. It can be explicitly described as  $p*q = tp(gh/M)$ where $g$ realizes $p$ and $h$ realizes the unique heir of $q$ over $(M,g)$.  

It is worth mentioning the notion of a definable action of $G$ on a compact space $X$. It means an action of $G$ on $X$ (by homeomorphisms) with the property that for each $y\in X$ the map from $G$ to $X$ taking $y\in X$ to $gy$ is definable. Where a map $\phi$  from $G$ to the compact space $X$ is said to be definable if for any two closed disjoint subsets $C_{1}$, $C_{2}$ of $X$, the preimages $\phi^{-1}(C_{1})$, $\phi^{-1}(C_{2})$ are separated by a definable (in $M$) subset of $G$. 

When all types over $M$ are definable, then $(S_{G}(M), G, id_{G})$ is the universal definable $G$-ambit (in analogy with $\beta G$ being the universal $G$-ambit).  Moreover some/any minimal subflow ${\cal M}$ of $S_{G}(M)$ will be the universal minimal definable $G$-flow.   And  the Ellis group $p{\cal M}$,  for $p$ any idempotent in ${\cal M}$,  will be a basic invariant of the definable group $G$. 

Another basic invariant of $G$ is the compact group $G({\bar M})/G({\bar M})^{00}_{M}$, where $G({\bar M})_{M}^{00}$  is the smallest bounded index type-definable over $M$ subgroup of $G({\bar M})$.  This can also be described as the definable Bohr compactification of $G$. 

As already pointed out in \cite{Newelski} there is a natural surjective homomorphism from the Ellis group to the definable Bohr compactification. Newelski suggested that in tame contexts such as when $Th(M)$ has $NIP$, this is actually an isomorphism.  This was proved in \cite{C-S} when $G$ is definably amenable (and proved earlier in \cite{C-P-S} when $M$ is $o$-minimal and $G$ definably amenable). In \cite{GPPI} we showed that $SL(2,\R)$ as a group definable in $(\R,+,\times)$ gives a counterexample. And one of the points of the current paper is that $SL(2,\Q)$ provides another (counter)example.  In fact it will be somewhat more striking as the (definable) Bohr compactification of $SL(2,\Q)$ is trivial, whereas the (definable) Ellis group will be infinite.  As mentioned earlier, our counterexample is different from the ones provided by Corollary 0.3 of \cite{Krupinski-PillayII}, as we have $G = G^{000}$ (where $G = SL(2,\K)$ which is abstractly simple modulo its finite centre).

This paper will make use of some more model-theoretic machinery, around definable amenability and $f$-genericity, in a $NIP$ environment (bearing in mind that $Th(\Q,+,\times, -, 0, 1)$ has $NIP$, see Section 4.2 of \cite{Belair} for references). 

So let us now assume, in addition to $G$ being a group definable in $M$, that $Th(M)$ has $NIP$.  $G$ is said to be definably amenable if there is a left $G$-invariant Keisler measure $\mu$ on $G$; namely $\mu$ is a map from the Boolean algebra of subsets of $G$ definable in $M$ to the real unit interval $[0,1]$,  $\mu(\emptyset) = 0$, $\mu(G) = 1$, $\mu$ is finitely additive, and $\mu(gX) = \mu(X)$ for all definable $X$ and $g\in G$. 
Now $SL(2,\Q)$ (as a group definable in $(\Q,+,\times, -, 0, 1)$)  will {\em not} be definably amenable, for the same reason that $SL(2,\R)$ is not, see \cite{NIPI}.  But its constituents in the Iwasawa decomposition will be definably amenable. 

Let $\bar M$ be a  very saturated elementary extension of $M$.  By a global type of $G$ we mean a type $p(x)\in S_{G}(\bar M)$.  $p$ is {\em strongly $f$-generic}  if every left $G$-translate of $p$ is $Aut({\bar M}/N)$-invariant (i.e. does not fork over $N$) for some small $N\prec \bar M$  depending only on $p$.  The existence of a strongly $f$-generic type is equivalent to definable amenability of $G$ (\cite{NIPII}). Assuming $G$ to be definably amenable we can take as a definition of $p$ being {\em $f$-generic} that $Stab(p)$ is $G({\bar M})^{00}$, and this is implied by $p$ being strongly $f$-generic. (See Section 3 of \cite{C-S}.)

There are two extreme cases for a strongly $f$-generic global type $p(x)$. The first case is when $p$ and all of its left translates are definable over some small model $N$,  and the second case is when $p$ and all of its left translates are  finitely satisfiable over some small model $M$. In the second case every strongly $f$-generic type is finitely satisfiable in {\em any} small model,  and $G$ is what is called an $fsg$ group (group with finitely satisfiable generics).  In this $fsg$ case, the $f$-generic types, strongly $f$-generic types, and generic types coincide. Here a (left) generic formula is one such that finitely many (left) translates cover the group, and a (left) generic type is one all of whose formulas are (left) generic. 

Finally let us remark that the sister paper \cite{GPPI} on $SL(2,\R)$ was subsequently extended in various ways; first to a larger class of semialgebraic real Lie groups, and secondly to arbitrary real closed base fields in place of $\R$.  See \cite{G. Jagiella} and \cite{Y-L}.  The $p$-adic version should be able to be extended similarly.   Also the analysis we give in this paper could be situated in a more general  environment of definable groups $G$  in $NIP$ structure, where $G$ has a nice ``abstract" Iwasawa decomposition.

\section{Ingredients and building blocks}
Our model-theoretic analysis of $SL(2,\R)$ (acting on its space of types) in \cite{GPPI} made heavy use of the Iwasawa decomposition $SL(2,\R) = K\cdot B(\R)^{0}$  with $K$ the maximal compact subgroup $SO(2,\R)$ and $B(\R)^{0}$ the (real) connected component of the Borel subgroup of upper triangular matrices. Note that both $K$ and $B(\R)^{0}$ are connected and  also have trivial intersection. 

The Iwasawa decomposition for $SL(2,\Q)$ on the other hand has the form $K\cdot B(\Q)$ where $K$ is the maximal compact $SL(2,\Z)$ and $B$ is again the Borel subgroup of upper triangular $2$-by-$2$ matrices. $B(\Q)$ is itself is the semidirect product of the additive and multiplicative groups of $\Q$ (where the action is multiplication by the square).   See \cite{Bruhat-Tits}. So we start with the model-theoretic/dynamical analysis of these building blocks.

\subsection{The additive and multiplicative groups}

If we refer to a theory $T$ it will be $Th(\Q,+,\times,-,0,1)$.  Recall that $P_{n}(x)$ denotes the formula saying that $x$ is  an $nth$ power, and that $T$ has quantifier elimination after adding predicates for all $P_{n}$.  Before getting into details we recall, with references, some basic facts which will be used freely in this section and the rest of the paper.  First the topology on both the standard model $\Q$ and the saturated model $\K$ is the valuation topology.  The following can be found in (or easily deduced from) Section 1 of \cite{Macintyre} (Facts 1 to 3) and Section 2 of \cite{Belair} and make use of Hensel's Lemma.   The (nonzero) $nth$ powers form an open subgroup of finite index in the multiplicative group, and each coset contains representatives from $\z$ even with valuation $0$.  
It is clear that the partial type $\cap_{n}P_{n}(x)$ defines the ``connected component" $(\K^{*})^{0}$ of the multiplicative group $\K^{*}$ of $\K$.  So every  translate of $(\K^{*})^{0}$ can be (type)-defined over  $\z$ too.  

We first describe the complete $1$-types over the standard model $M = (\Q,+,\times, -, 0, 1)$.

\begin{Lemma}  The complete $1$-types over $M$ are precisely the following:
\newline
(a) The realized types  $tp(a/M)$ for each $a\in \Q$.
\newline
(b) for each $a\in \Q$ and coset $C$ of $(\K^{*})^{0}$ in $\K^{*}$  the type $p_{a,C}$ saying that $x$ is infinitesimally close to $a$ (i.e. $v(x-a) > n$ for each $n\in \N$),  and  $(x-a)\in C$ (note this implies $x\neq a$). 
\newline
(c) for each coset $C$ as above the type $p_{\infty,C}$ saying that $x\in C$ and $v(x)<n$ for all $n\in\z$. 
\end{Lemma}
\begin{proof}  (i) First we observe that  every nonrealized $1$-type over the standard model $M$ is either ``at infinity" namely contains the formulas $v(x)< n$ for all $n\in \z$, or is infinitesimally close to some $a\in \Q$, namely contains the formulas $v(x-a)> n$ for all $n\in \z$.  This depends on compactness of ``balls" defined by $v(x)\geq n$ in the standard model, and is not true over a saturated model, as we remark in 2.2(iii). 

Next we show that each purported complete $1$-type over $M$ described in (b) is consistent.  Fix $a\in \Q$. It suffices, by compactness to show that for every $n$, $k$, and coset $C$ of the $kth$ powers, $v(x-a) > n\wedge (x-a)\in C$ has a solution.  Choose an element $b$  in $C\cap\Z$ (as mentioned earlier we can find one). Let $r$ be a natural number such that $rk>n$. Then $bp^{rk}\in C$ and $v(bp^{rk}) > n$. Let $x = a + bp^rk$, then $(x-a)$ has value $>n$ and is in $C$. 

A similar argument shows consistency of any type of kind (c).  

Note that for any complete type $p(x)$ over $M$ and any $a\in M$, $p$ has to choose some coset $C$ of $(K^{*})^{0}$ such that ``$(x-a)\in C$" is in $p$.  

So it remains to show completeness of the $p_{a,C}$ and $p_{\infty,C}$.  We will do the case of $p_{0,C}$ from (b). (The general case of (b) is similar, by expanding polynomials around $a$.) To show completeness of $p_{0,C}$ it is enough, by quantifier elimination, to show that  $p_{0,C}$ decides each formula of the form $P_{n}(f(x))$ where $f(x)$ is a polynomial over $\Q$. 
Suppose $f(x) = a_{i}x^{i} + a_{i+1}x^{i+1} + \cdots +a_{m}x^{m}$ where   $a_{i}\neq 0$.   Let $c$ realize $p_{0,C}$. Then $c^{- i}f(c) = a_{i} + a_{i+1}c  + .. + a_{m}c^{m-i}$.
As $v(a_{i+1}c + .. + a_{m}c^{m-i}) > \z$, and each (multiplicative) coset of the $nth$ powers is open, it follows that $c^{-i}f(c)$ and $a_{i}$ are in the same coset of $P_{n}$.  But the coset of $P_{n}$ that $c^{-i}$ is in is determined by $c$ realizing $p_{0,C}$. Hence the coset of $P_{n}$ in which $f(c)$ lives is also determined, by $c$ realizing $p_{0,C}$, as required. 

Finally we show completeness of $p_{\infty,C}$ from (c).  Consider again a formula $P_{n}(f(x))$, (with $f(x)$ over $\Q$)  and we want it to be decided by $p_{\infty,C}$. Again let $f(x) = a_{i}x^{i} + ...+ a_{m}x^{m}$ with $a_{i}\neq 0$ and $a_{m}\neq 0$.   Let $c$ realize $p_{\infty,C}$.   Now we consider $c^{-m}f(c) =  a_{i}c^{i-m} + .. +a_{m-1}c^{-1} + a_{m}$. 

Now $v(a_{i}c^{i-m} + ... + a_{m-1}c^{-1}) > \z$. So again as cosets  of  the $nth$ powers are open in the multiplicative group,  it follows that $c^{-m}f(c)$ and $a_{m}$ are in the same coset of $P_{n}$. Hence again the coset of $f(c)$ modulo $P_{n}$ is determined by $c$ realizing $p_{0,C}$. 




\end{proof}

\begin{Rmk} (i) The lemma shows that the definable (with parameters) subsets of $\Q$ are precisely given by (finite) Boolean combinations of formulas $x=a$, $v(x-a)\geq n$ and $(x-a)\in C'$, for $a\in \Q$, $n\in \z$ and $C'$ a coset of the $nth$ powers.  
\newline
(ii)  An identical proof to the above shows that working now over the saturated model ${\bar M} = (\K,....)$, if $p(x)\in S_{1}({\bar M})$ is a nonrealized complete $1$-type ``at infinity", namely containing $v(x)< \Gamma$, then for some coset $C$ of $(\K^{*})^{0}$,  $p$ is axiomatized by $v(x)<\Gamma$ together with $x\in C$. Similarly if $p(x)\in S_{1}(\bar M)$ is nonrealized and says that $v(x-a) > \Gamma$ for some $a\in \K$ then for some $C$ as before  $p$ is axiomatized by $x\neq a$ and $v(x-a)>\Gamma$ together with $(x-a)\in C$. 
\newline
(iii) There will be nonrealized $1$-types over ${\bar M}$ not accounted for in (ii), but we do not have to describe them precisely for the purposes of this paper. 
\end{Rmk}

We start with the additive group. We consider $S_{1}(M)$ as a $(\Q,+)$-flow.  We could and should write it as $S_{\Ga}(M)$ but this is too much notation. 
\begin{Prop} (i) Each type $p(x)\in S_{1}(M)$ of kind (c) is  invariant (under the action of $(\Q,+)$), and these account for all the minimal subflows of $S_{1}(M)$.
\newline
(ii) The global heirs of the types in (i) are precisely the global (strongly) $f$-generics of $(\K,+)$, and are all definable, and invariant under $(\K,+)$. 
\newline
(iii) $(\K,+) = (\K,+)^{0} = (\K, +) ^{00}$.  
\end{Prop}
\begin{proof} (i) Let $a$ realize $p_{\infty,C}$.  Then clearly for $b\in \Q$, $v(a+b)<\z$. On the other hand, for $b\in \Q$, $(a+b)/a = 1 +(b/a)$  and note that $v(b/a) > \z$.  As the group of $nth$ powers in $\Q$ is open for all $n$, it follows that $a+b$ and $a$ are in the same coset of the $nth$ powers for all $n$, and so in particular $a+b\in C$. We have shown that $p_{\infty, C}$ is fixed under addition by elements of $\Q$, as required. 
\newline
If $q(x)\in S_{1}(M)$ is arbitrary note that the closure of the orbit (under $(\Q,+)$) of $q$ always contains a ``type at infinity" namely a type of kind (c). Hence the only minimal subflows of $S_{1}(M)$ are those of the form $\{p\}$ for $p$ of kind (c).
\newline
(ii) and (iii).  Let $q$ be a global heir of a type $p_{\infty,C}$ of kind (c). Then $q$ is definable over $M$ and $Stab(q) = (\K,+)$.  This already shows that 
that $(\K,+) = (\K,+)^{00} (= (\K,+)^{000})$ (because any global type $1$-type determines a coset of  $(\K,+)^{000}$). Conversely suppose $q(x)\in S_{1}(\bar M)$ is an $f$-generic. Then by what we have just said, together with the fact that $(\K,+)$ is definably amenable, since it is abelian,  $q$ must be  $(\K,+)$-invariant.  We claim first  that $q$ must be a  ``type at infinity".  For otherwise  ``$v(x)\geq \gamma$" is in $q(x)$ for some $\gamma$ in the value group $\Gamma$ of $\K$.  Then for  $b\in \K$ with $v(b) < \gamma$, $q + b \neq q$, a contradiction.   So $q$ is a type at infinity as claimed. By Remark 2.2(ii), $q$ is axiomatized by $v(x)<\Gamma$ together with $x\in C$ for some coset $C$ of $(\K^{*})^{0}$. But then clearly $q$ is definable over $\Q$ and so is the heir of $p_{\infty, C}$. 
\end{proof}

Now for the multiplicative group. $S_{\Gm}(M)$ denotes the space of complete types over $M$ concentrating on $\Gm$, namely all complete $1$-types except for $x=0$.
$\Gm(\Q)$ is just the multiplicative group $(\Q^{*},\times)$ and $S_{\Gm}(M)$ is a $\Gm(\Q)$-flow.
\begin{Prop} (i) $S_{\Gm}(M)$ has two minimal subflows, the collection of types of kind (b) with $a=0$, namely $P_{0} = \{p_{0,C}:C$ coset of $(\K^{*})^{0}\}$ and the collection of types of kind (c), namely $P_{\infty} = \{p_{\infty,C}: C$ coset of $(\K^{*})^{0}\}$. 
\newline
(ii) The global heirs of the types of the types mentioned in (i) are precisely the global (strongly) $f$-generic types of $\Gm$, all of which are definable.  Moreover the orbit of each such type under $\K^{*}$ is closed. 
\newline
(iii) $(\K^{*})^{00} = (\K^{*})^{0}$. 
\end{Prop} 
\begin{proof} (i) Fix any $p_{0,C}\in P_{0}$. Then it is clear that the closure of its orbit under $\Q^{*}$ equals $P_{0}$. Likewise for $P_{\infty}$. Hence $P_{0}$ and $P_{\infty}$ are minimal subflows. On the other hand it is clear that for any $q(x)\in S_{\Gm}(M)$, the closure of the orbit of $q$ under $\Q^{*}$ intersects both $P_0$ and $P_{\infty}$. Whence  $P_{0}$ and $P_{\infty}$ are the only minimal subflows of $S_{\Gm}(M)$. 
\newline
(ii)  and (iii). Fix $p_{0,C}$.  Let $p_{0,C}'$ be its (unique) global heir. Note that $p_{0,C}'$ is axiomatized again by $v(x) > \Gamma$ and $x\in C$.  It is clear that the the stabilizer of $p_{0,C}'$ (with respect to the action of $\K^{*}$)  is precisely $(\K^{*})^{0}$, whereby $p_{0,C}'$ is $f$-generic. On the other hand the orbit of $p_{0,C}'$ under $\K^{*}$ is precisely $P_{0}'$ the collection of global heirs of the types in $P_{0}$.  Hence $p_{0,C}'$ is also strongly $f$-generic.  In a similar fashion the unique global heir of each $p_{\infty,C}$ is strongly $f$-generic. We have shown that the types in $P_{0}'$ and the analogue $P_{\infty}'$ have stabilizer $(\K^{*})^{0}$, are definable over $M$ and are all strongly $f$-generic.
This already shows that $(\K^{*})^{00} = (\K^{*})^{0}$.  Bearing in mind Remark 2.2 (ii), we have shown that global types at infinity or infinitesimally close to $0$ are  (strongly) $f$-generic. It is easy to see that these are the only (strongly) $f$-generics;  suppose $\gamma\in \Gamma$ is positive, and $q(x)$ is a global $1$-type implying that $-\gamma < v(x) < \gamma$.  We can find $g\in (\K^{*})^{0}$ with $v(g) > 2\gamma$. But then $gq$ implies $v(x) > \gamma$, so $q$ is not invariant under multiplication by $(\K^{*})^{0}$, so could not be $f$-generic. 
\end{proof} 

\begin{Rmk} (i) So note that $\K^{*}/(\K^{*})^{00}$ is $\hat \z$, which is not a compact $p$-adic Lie group. 
\newline
(ii) In fact the valuation homomorphism $v:\K^{*}\to \Gamma$ induces an isomorphism between $\K^{*}/(\K^{*})^{00}$ and $\Gamma/\Gamma^{00}$. 
\end{Rmk} 
\begin{proof} We have seen above that $(\K^{*})^{00} = (\K^{*})^{0}$ which is obviously the intersection of the $(\K^{*})^{n}$. Notice that $v$ takes $(\Q^{*})^{n}$ onto $n\z$ (as $v(p^{nk}) = nk$). So $v$ takes $(\K^{*})^{n}$ onto $n\Gamma$, so establishes an isomorphism between 
$\K^{*}/(\K^{*})^{n}$ and $\Gamma/n\Gamma = \z/n\z$.  This induces an isomorphism between the inverse limit of the $\K^{*}/(\K^{*})^{n}$ and $\hat\z$. 
\end{proof} 

Finally we discuss the  additive and multiplicative groups of the valuation ring $\V$. $\V(M)$ is $(\Z,+)$, and $\V^{*}(M) = (\Z^{*},\times)$, where remember that $\V^{*}$ is defined by $v(x) = 0$.  These groups $(\Z,+)$ and $(\Z^{*},\times)$ are compact groups definable in $\Q$, so by Corollary 2.3 of \cite{O-P},  $(\V,+)$ and $(\V^{*},\times)$ are $fsg$ groups, which have been studied intensively.  We record the basic facts, leaving details to the interested reader. 
\begin{Prop} (i) The universal definable minimal flow of $(\Z,+)$ is the space $S_{\V, na}(M)$ of nonalgebraic types concentrating on $\V$, which are precisely the types in (b) above for $a\in \Z$. 
\newline
(ii) The global coheirs of the types in (i) are precisely the global (strongly)  $f$-generic types of $\V$ which coincide with the generic types of $\V$.  
\newline
(iii) The orbits in $S_{\V,na}(M)$ are indexed by the multiplicative cosets $C$, namely a typical orbit is of the form  $\{p_{a,C}: a\in \Z\}$.
\end{Prop}

\begin{Rmk} Compare to the case where $M = (\R,+,\times)$ and $G$ is the circle group  ($SO_{2}$, or $[0,1)$ with addition mod $1$). The set of nonalgebraic types is the unique minimal flow and there are two orbits, infinitesimal  to the left, and infinitesimal to the right (of each point of $G$ in the standard model). 
\end{Rmk}

\begin{Prop} (i) The universal definable minimal flow of $(\Z^{*},\times)$ is the space $S_{\V^{*}, na}(M)$ of nonalgebraic types concentrating on $\V^{*}$, namely the types $p_{a,C}$ with $v(a) = 0$. 
\newline
(ii) The global (strongly) $f$-generic types are precisely the coheirs of these types and they coincide with the global generic types. 
\newline
(iii) The $\Z^{*}$-orbits in  $S_{\V^{*}, na}(M)$ are precisely the sets $\{p_{a,aC}: a\in \Z^{*}\}$ for $C$ a coset of $(\K^{*})^{0}$.

\end{Prop}

\subsection{The Borel subgroup}
The Borel subgroup $B$ of $SL(2,-)$ is the group of upper triangular $2$-by-$2$ matrices of determinant $1$. 

So $B(\K)$ is the subgroup of $SL(2,\K)$ consisting of matrices $\begin{bmatrix} a& c\\ 0& {a^{-1}}    \end{bmatrix}$ where $a\in \K^{*}$ and $c\in \K$. There is no harm in identifying the matrix $\begin{bmatrix} a& c\\ 0& {a^{-1}}    \end{bmatrix}\in B(\K)$ with the pair $(a,c)\in \K^*\times \K$.  Likewise for $B(\Q)$.  Note that with this notation the product $(a,c)(\alpha,\beta)$ equals $(a\alpha, a\beta + c\alpha^{-1})$.

\begin{Lemma}
$B(\K)^{00} = B(\K)^{0} = \{(a,c): a\in (\K^{*})^{0}, c\in \K\}$. 
\end{Lemma}
\begin{proof}
$B(\K)$ maps onto $\K^{*}$ with kernel $(\K,+)$.  So the result follows from Proposition 2.3 (iii) and Proposition 2.4(iii). 
\end{proof}

Rather than describe all the global $f$-generic types of $B(\K)$ we will choose one, as follows.  Let $C_{0}$ denote $(\K^{*})^{0}$, the connected component of the multiplicative group. Then $p_{0,C_{0}}'$, the unique global heir of $p_{0, C_{0}}$,  is a global $f$-generic of $\K^{*}$, and likewise  $p_{\infty, C_{0}}'$, the unique global heir of $p_{\infty, 
C_{0}}$ is a global $f$-generic of $(\K,+)$.  Let $\alpha$ realize $p_{0,C_{0}}'$ and $\beta$ realize $p_{\infty,C_{0}}'$ such that $tp(\alpha/{\bar M},\beta)$ is finitely satisfiable in ${\bar M}$. 

We let $\bar{p_0}=\tp((\alpha, \beta)/{\bar M})\in S_B(\bar M)$, and let $p_{0} = tp((\alpha, \beta)/M)$ be its restriction to $M$. Note that this is {\em new} notation which will be used in Section 3 too.

\begin{Lemma}
$\bar{p_0}\in S_B(\bar M)$ is  a global (strongly) f-generic type of  $B(\K)$, every (left) $B(\K)$-translate of which is definable over $M$.   
\end{Lemma}

\begin{proof}
We note first that $\bar{p_0}$ is left $B(\K)^{0}$-invariant:  let $(a,c)\in B(\K)^{0}$, which by Lemma 2.9 means that $a\in (\K^{*})^{0}$. 
Now  $(a,c)(\alpha,\beta)=(a\alpha, a\beta+c\alpha^{-1})$. As  $a\in{\K^*}^0$, then $\tp(a\alpha/\K)=\tp(\alpha/\K)$. On the other hand,   $\beta$ also realizes a global $f$-generic type of the multiplicative group. So $tp(a\beta/{\bar M}) = tp(\beta/{\bar M})$.  Also $\tp(a\beta/{\bar M}, c\alpha^{-1}, a\alpha)$ realizes the unique heir of $tp(\beta/{\bar M})$ so as the latter is an $f$-generic of the additive group which is connected, we have that $tp(a\beta + c\alpha^{-1}/{\bar M}, c\alpha^{-1}, a\alpha)$ is an heir of $tp(\beta/{\bar M})$. It follows from all of this that
$tp((a\alpha, a\beta+c\alpha^{-1})/{\bar M}) = tp ((\alpha, \beta)/{\bar M})$ as required. 

\vspace{2mm}
\noindent
As $tp(\alpha/{\bar M})$ is definable over $M$ and $tp(\beta/{\bar M},\alpha)$ is the heir of $tp(\beta/{\bar M})$ which is definable over $M$, then $tp(\alpha,\beta/{\bar M})$ is definable over $M$. Using a similar argument as in the first paragraph, every left $B(\K)$-translate of $\bar{p_0}$ is definable over $M$. 
 
\end{proof}


\begin{Cor}\label{closed-orbit-B}
(i) The $B(\K)$-orbit of $\bar{p_0}$ is closed, and hence is a minimal $B(\K)$-subflow of $S_{B}({\bar M})$.  
\newline
(ii) Let $\bar{\cal J}$ denote the $B(\K)$-orbit of $\bar{p_0}$ and $\cal J$ the closure of the  $B(M)$-orbit of $p_{0}$. Then the restriction to $M$ map gives a homeomorphism between $\bar{\cal J}$ and $\cal J$, and  $\cal J$ is a minimal subflow of $S_{B}(M)$.
\newline
(iii) $\cal J$ is a subgroup of $(S_{B}(M),*)$,  is isomorphic to $B(\bar M)/B(\bar M)^{0}$, and is the Ellis group of the dynamical system $(B(M),S_{B}(M))$.  Moreover $p_{0}$ is an idempotent in $(S_{B}(M),*)$. 
\end{Cor}
\begin{proof}
(i). This  follows by using the proof of  Lemma 1.15 of \cite{P-Y}. More specifically it is proved there that in the $NIP$ environment, a global definable $f$-generic type $p$ is almost periodic by showing that in fact the orbit of $p$ is closed. 
\newline
(ii). We have seen in Lemma 2.10 that every $p\in {\bar{\cal J}}$ is the unique heir of its restriction to $M$.  Hence the restriction to $M$ map, $\pi$, which is a continuous map between ${\bar{\cal J}}$ and its image $\pi({\bar{\cal J}})$ is a bijection, hence a homeomorphism.  Now it is fairly easy to see directly that $\pi({\bar{\cal J}})$ is a minimal $B(M)$-subflow of $S_{B}(M)$, although we can also appeal to the general result Corollary 4.7 of \cite{Simon} to see this. As $\cal J$ is a closed $B(M)$-subflow of $\pi({\bar{\cal J}})$, it follows that they are equal, and we obtain all of (ii).
\newline
(iii). The natural map from $S_{B}(\bar M)$ to (the profinite group) $B({\bar M})/B({\bar M})^{0}$,  is continuous. Moreover this map induces a 
bijection hence homeomorphism between $\bar{\cal J}$ and $B({\bar M})/B({\bar M})^{0}$.  Composing with the homeomorphism between $\cal J$ and $\bar{\cal J}$ gives a homeomorphism $\theta$ say between ${\cal J}$ and $B({\bar M})/B({\bar M})^{0}$. It is clear that this is also an isomorphism of semigroups, whereby $({\cal J},*)$ is already a group, so must concide with the Ellis group $(u*{\cal J},*)$ ($u$ an idempotent of ${\cal J}$).  As $\bar{p_{0}}/B({\bar M})^{0}$ is  the identity of $B({\bar M})/B({\bar M})^{0}$ by Lemma 2.9, it follows that $p_{0}$ is an (in fact the) idempotent of ${\cal J}$.

\end{proof}

\subsection{The maximal compact subgroup}
As already remarked a  maximal compact subgroup of $SL(2,\Q)$ is $SL(2,\Z)$. We refer to this group as $K$ and sometimes, by abuse of language, we also let $K$  denote the defining formula. So $K(\bar M)$ is $SL(2,\V)$, and  $S_{K}(M)$, $S_{K}(\bar M)$ denote the corresponding type spaces. (The notation $\V$ for the valuation ring in the saturated model $\bar M$ was introduced in Section 2.1.) We have the standard part map $st:SL(2,\V) \to SL(2,\Z)$ the kernel of which is (by definition) the infinitesimals. By Corollary 2.4 of \cite{O-P}, this kernel coincides with  $SL(2,\V)^{00}$.  (Note that as $SL(2,\Z)$ is profinite, this group of infinitesimals is an intersection of definable groups, so coincides with $SL(2,\V)^{0}$.)

From Corollary 2.3 of \cite{O-P},   $K$ is an $fsg$ group. In particular, we have

\begin{Fact} (i)  Left and right generic definable subsets of $K(\bar M)$ coincide and are all satisfiable in $M$.
\newline
(ii)  There exist left generic types in $S_{K}(\bar M)$, which by (i) coincide with right generic types. 
\newline
(iii) The unique minimal $K$-subflow of $S_{K}(M)$ is the set $\cal I$ of generic types over $M$.
\newline
(iv)  Likewise the unique minimal subflow ${\cal I}'$ of $S_{K}({\bar M})$ is the set of global generic types, each such global generic type $q$ being the unique coheir of $q|M\in {\cal I}$. 
\newline
(v) The standard part map $st$ induces  an isomorphism (in fact homeomorphism) between $K(\bar M)/K({\bar M})^{00}$ and $SL(2,\Z)$
\end{Fact}

Recall that in the current situation where all types over $M$ are definable, we have the semigroup operation $*$ on $S_{K}(M)$, and $\cal I$ is a left ideal under $*$.

From Theorem 3.8 of \cite{Pillay-top} for example, we see that the Ellis group of the action of $SL(2,\Z)$ on $S_{K}(M)$ is canonically isomorphic to $K/K^{00} = SL(2,\Z)$.   With notation as in Fact 2.12 this Ellis group is $u*{\cal I}$ for some/any idempotent in ${\cal I}$.  Different choices of $u$ give isomorphic groups and the collection of such $u*{\cal I}$ partitions ${\cal I}$. 
We will elaborate slightly on these basic facts.

\begin{Lemma}\label{q_0t=tq_0}
\begin{enumerate}
\item $\cal I$ is a two-sided ideal of $(S_K(M),*)$.   
\item For any $q\in \cal I$, $q*S_K(M)$ is the copy of the  Ellis group which contains $q$.

\end{enumerate}
\end{Lemma}
\begin{proof}
\begin{enumerate}
\item  Let $q\in \cal I$ and $p\in S_K(M)$.  Let $b\in K(\bar M)$ realize $p$ and let $a$ realize the unique coheir of $q$ over $\bar M$.  Then $tp(ab/M)$ realizes $q*p$.  On the other hand, $tp(a/\bar M)$ is right generic, whereby $tp(ab/\bar M)$ is also right generic, so by Fact 2.12 (iv), $q*p =tp(ab/M)\in {\cal I}$. 

\item Again let $q\in \cal I$. Let $E\subseteq \cal I$ be the copy of the Ellis group which contains $q$, and let $q_{0}$ be an idempotent in $E$. 
Then $ q*S_{K}(M) = (q_{0}*q)*(S_{K}(M)) = q_{0}*(q*S_{K}(M))\subseteq q_{0}*{\cal I}$ (using part 1.) $\subseteq q_{0}*q*{\cal I} \subseteq q*{\cal I} \subseteq q*S_{K}(M)$. This shows that $q*S_{K}(M) =  q_{0}*{\cal I}$ which equals $E$.

\end{enumerate}

\end{proof}

As usual for $x,y$ in a given group $G$, $x^y$ denotes the conjugate $yxy^{-1}$ of $x$ by $y$ and the notation extends naturally to subsets $X$ of $G$ in place of $x\in G$. 

In our context  $G = SL(2,\K)$ and $K(\bar M)$ is $SL(2,\V)$.
\begin{Lemma} $(K(\bar M)^{0})^{g} = K(\bar M)^{0}$ for all $g\in SL(2,\Q) = G(M)$. 
\end{Lemma}
\begin{proof} We know that $K(\bar M)^{0}$ is the kernel of $st:K(\bar M) \to K(M)$, so equal to $\cap_{V}V(\bar M)$ where $V$ ranges over open semialgebraic neighbourhoods of the identy in $K(M) = SL(2,\Z)$. But $SL(2,\Z)$ is an open (semialgebraic) subgroup of $SL(2,\Q)$, so $K(\bar M)^{0} = \cap_{V} V(\bar M)$ where $V$ ranges over  open semialgebraic neighbourhoods of the identity in $SL(2,\Q)$.
But clearly the family of open semialgebraic neighbourhoods of the identity in $SL(2,\Q)$ is invariant under conjugation  by elements of $SL(2,\Q)$. Hence the lemma follows.

\end{proof}

\begin{Cor}  Let $g\in G(\bar M)$ and $t\in K(\bar M)^{0}$ be such that $tp(g/t,M)$ is finitely satisfiable in $M$. Then $t^{g}\in K({\bar M})^{0}$.  If in addition, $tp(t/M)$ is a generic type of $K$ then so is $tp(t^{g}/M)$. 

\end{Cor}
\begin{proof}  The first sentence is fairly immediate from Lemma 2.14: if by way of contradiction $t^{g}\notin V(\bar M)$ for some open semialgebraic neighbourhood of the identity of $K(M)$, then there is $g_{1}\in G(M)$ such that $t^{g_{1}}\notin V(\bar M)$, contradicting Lemma 2.14.

The second sentence follows from the fact that the set of generic types in $S_{K}(M)$ is closed.

\end{proof}

\section{$SL_{2}(\Q)$}
We use the above material to describe the minimal definable universal subflow of $SL(2, \Q)$ as well as its Ellis group.  We first identify the minimal subflow, see Theorem 3.4 below. 

\subsection{Minimal subflow of $(G(M), S_{G}(M))$}

The  Iwasawa decomposition of $SL(2,\Q)$ is  $B(\Q)\cdot SL(2,\Z)$, namely every element of $SL(2,\Q)$ can be written as a product $ht$ with $h\in B(\Q)$ and $t\in SL(2,\Z)$  (and also as a product $t_{1}h_{1}$ with $t_{1}\in SL(2,\Z)$ and $h_{1}\in B(\Q)$).  However, in contradistinction to the Iwasawa decomposition for real Lie groups, there is a large intersection of the constituents; namely $B(\Q)\cap SL(2,\Z) = B(\Z) = $

\[\left\{\left[\begin{array}{cc}                                                                                        a& c\\0&a^{-1}                                                                                       \end{array}\right]|a\in \Z^* \text{~and~} c\in\Z\right\}.\]

\vspace{2mm}
\noindent
We recall the notation from the previous sections: $K(M) = SL(2, \Z)$ is the maximal compact subgroup of  $SL(2,\Q)$, and $\cal I$ is the unique minimal subflow of the flow $(K(M), S_{K}(M))$. We fix a generic type $q_{0}\in S_{K}(M)$ which concentrates on $K^{0}$.  
$p_{0}$ is the restriction to $M$ of the global $f$-generic type $\bar{p_{0}}$ of $B(\K)$, and $\cal J$ is the minimal subflow of $(B(M), S_{B}(M))$ containing $p_{0}$, as in subsection 2.2. 

\begin{Lemma}  ${\cal I}*{\cal J} \subseteq S_{G}(M)*q_{0}*p_{0}$. 
\end{Lemma}
\begin{proof}  We have to show that for any $q_{1}\in \cal I$ and $p_{1}\in \cal J$, there is $s\in S_{G}(M)$ such that $s*q_{0}*p_{0} = q_{1}*p_{1}$.  

Let $p'\in S_{B}(M)$ be such that $p'*p_{0} =p_{1}$.  (Because $p_{0},p_{1}\in \cal J$ which is a minimal subflow of $S_{B}(M)$ so of the form $S_{B}(M)*p_{0}$.)   Now let $s = q_{1}*p'$. 

Then $$s*q_{0}*p_{0} = q_{1}*p'*q_{0}*p_{0}  = tp(t_{0}hth_{0}/M) = tp(t_{0}t^{h}hh_{0}/M)$$

where $t_{0}$ realizes $q_{1}$, $h$ realizes the (unique) heir of $p'$ over $M,t_{0}$, $t$ realizes the unique heir of $q_{0}$ over $M,t_{0},h$ and $h_{0}$ realizes the unique heir of $p_{0}$ over $M,t_{0},h,t$. We may assume that $t_{0}, h, t$ are in $SL(2, \K)$ and that $h_{0}$ realises the unique heir of $p_{0}$ over ${\bar M}$. 

By Lemma 2.10, $tp(hh_{0}/{\bar M})$ is definable over $M$, and note that $tp(hh_{0}/M) = p'*p_{0} = p_{1}$.

On the other hand, by Corollary 2.15, $t^{h}\in K^{0}$. As $q_{1}\in {\cal I}$ and $t_{0}$ realizes the unique coheir of $q_{1}$ over $M,t^{h}$, we have that $tp(t_{0}t^{h}/M) = q_{1}$.

Hence $tp(t_{0}t^{h}hh_{0}/M) = q_{1}*p_{1}$ as required.

\end{proof} 

\begin{Lemma}  $S_G(M)*q_{0}*p_{0} = cl({\cal I}*{\cal J})$
\end{Lemma}
\begin{proof}
 The previous lemma together with the fact that $S_{G}(M)*q_{0}*p_{0}$ is closed shows that the RHS is contained in the LHS.

For the converse,  we will show that the $G(M)$ orbit of $q_{0}*p_{0}$ is contained in ${\cal I}*{\cal J}$ which suffices, by taking closures, to see that the LHS is contained in the RHS.

So let $g\in G(M)$ and write $g = th$ with $t\in K(M) = SL_{2}(\Z)$ and $h\in B(M) = B(\Q)$.  Then

$$ (th)(q_{0}*p_{0}) = tq_{0}^{h}*hp_{0} = tp(tt_{0}^{h}hh_{0}/M)$$

where $t_{0}$ realizes $q_{0}$ and $h_{0}$ realizes the unique heir of $p_{0}$ over $M,t_{0}$. By the choice of $\cal J$,  $tp(hh_{0}/M)\in {\cal J}$.  Clearly (or by 2.15), $tp(t_{0}^{h}/M)\in {\cal I}$, as is $tp(tt_{0}^{h}/M)$.  Now as $tp(h_{0}/M,t_{0})$ is an heir of its restriction to $M$, also $tp(hh_{0}/M,tt_{0}^{h})$ is an heir of its restriction to $M$, so $tp(tt_{0}^{h}hh_{0}/M) \in {\cal I}*{\cal J}$, so by the displayed equation above 
\newline
$g(q_{0}*p_{0})\in {\cal I}*{\cal J}$, as required.

\end{proof}

\begin{Lemma} $S_{G}(M)*q_{0}*p_{0} \subseteq S_{K}(M)*{\cal J}$. Namely every $s*q_{0}*p_{0}$ (with $s\in S_{G}(M)$)  is of the form $r*p$ with $r\in S_{K}(M)$ and $p\in {\cal J}$.

\end{Lemma}
\begin{proof} Let $s = tp(th/M)$ where $t\in K(\bar M) = SL(2, \V)$ and $h\in B(\bar M) = B(\K)$.  Then 

$$ s*q_{0}*p_{0} = tp(tht_{0}h_{0}/M) = tp(tt_{0}^{h}hh_{0}/M)$$

where  $t_{0}$ realizes the unique heir of $q_{0}$ over $(M,t,h)$ and $h_{0}$ realizes the unique heir of $p_{0}$ over $\bar M$, namely $\bar{p_{0}}$. 
Again, $tp(hh_{0}/{\bar M})$ is definable over $M$ (by Lemma 2.10), and  $tp(hh_{0}/M)\in {\cal J}$ (by Corollary 2.11). Moreover by 2.15, $t_{0}^{h}\in K$ and so also $tt_{0}^{h}\in K$. Thus
$tp(tt_{0}^{h}hh_{0}/M)\in S_{K}(M)*{\cal J}$ as required. 

\end{proof}

\begin{Theorem} (i) $cl({\cal I}*{\cal  J})$ is a minimal subflow of the flow $(G(M), S_{G}(M))$.  
\newline
(ii) Moreover $q_{0}*p_{0}$ is an idempotent in this minimal flow. 
\end{Theorem}
\begin{proof} (i). By Lemma 3.2,  $cl({\cal I}*{\cal  J})$ is a $G(M)$-flow. As any point in $cl({\cal I}*{\cal J})$ is of the form $s*q_{0}*p_{0}$ by Lemma 3.2, and the closure of the $G(M)$-orbit of this $s*q_{0}*p_{0}$ is precisely 
$S_{G}(M)*s*q_{0}*p_{0}$, it suffices to prove:
\newline
{\em Claim.}  For any $s\in S_{G}(M)$, ${\cal I}*{\cal J} \subseteq S_{G}(M)*s*q_{0}*p_{0}$.
\newline
{\em Proof of claim.} 
Fix $s\in S_{G}(M)$. By the previous lemma, let $q'\in S_{K}(M)$, and $p_{1}\in {\cal J}$ be such that 
\newline
(1)   $s*q_{0}*p_{0} = q'*p_{1}$, and note that by Lemma 2.10 and Corollary 2.11 (ii) the unique global heir of $p_{1}$ is a strong $f$-generic of $B$ every translate of which is definable over $M$.  
\newline
We can easily find $q_{1}\in {\cal I}$ such that 
\newline 
(2) $q_{1}*q'\in K^{0}$ (in the obvious sense that some/any realization is in $K^{0}$). 

\vspace{2mm}
\noindent
Now, let $q\in {\cal I}$ and $p\in {\cal J}$ and we want to show that $q*p \in S_{G}(M)*s*q_{0}*p_{0}$. 
\newline
Let $p'\in S_{B}(M)$ be such that:
\newline 
(3) $p'*p_{1} = p$, where $p_{1}$ is as in (1). 

\vspace{2mm}
\noindent
Now we compute $q*p'*q_{1}*q'*p_{1}$. Let $a$ realize $q$, $b$ realize the unique heir of $p'$ over $(M,a)$, $c$ realize the unique heir of $q_{1}*q'$ over $(M,a,b)$ and $d$ realize the unique heir of $p_{1}$ over $\bar M$ (so in particular over $(M,a,b,c)$). 
Then 
\newline
 (4)      $q*p'*q_{1}*q'*p_{1} = tp(abcd/M) = tp(ac^{b}bd/M)$

\vspace{2mm}
\noindent
Now by the property of $p_{1}$ in (1), $tp(bd/{\bar M})$ is definable over $M$. In particular (using (3))  $bd$ realizes the unique heir of $p$ over  $(M,ac^{b})$.   On the other hand, by 2.15 and (2), $c^{b}\in K^{0}({\bar M})$. As $tp(c/M,a,b)$ is definable over $M$, and $tp(b/M,a)$ is definable over $M$, $tp(a/M,b,c)$ is finitely satisfiable in $M$ (and moreover realizes the {\em unique} coheir over $(M,b,c)$ of $q$, as all types over $M$ have unique heirs).  As the stabilizer (inside $K$) of the global coheir of $q$ is $K^{0}$, it follows that $tp(ac^{b}/M) = q$. So we conclude that
\newline
(5) $tp(ac^{b}bd/M) = tp(ac^{b}/M)*tp(bd/M) = q*p$. 
\newline
By (4) and (5) $q*p = r_{1}*(q'*p_{1})$   where $r_{1} = q*p'*q_{1}\in S_{G}(M)$. So by (1) $q*p = r_{1}*s*q_{0}*p_{0}$ giving the claim. 
\newline
{\em End of Proof of claim.} 
\newline
This finishes the proof of  (i).

\vspace{2mm}
\noindent
(ii) is an easy computation, bearing in mind the techniques above, which we carry out below. 

We want to show that
$$ q_{0}*p_{0}*q_{0}*p_{0} = q_{0}*p_{0}$$

The left hand side is $tp(tht_{0}h_{0}/M)$, where $t$ and $t_{0}$ realize $q_{0}$, $h$ and $h_{0}$ realize $p_{0}$, and $tp(t/M,h_{0}, t_{0}, h_{0})$ is the coheir of $q_{0}$ etc.  We will slightly adapt the proof of Lemma 3.3.  First rewrite this left hand side as $tp(t(t_{0}^{h})hh_{0}/M)$.  Conclude from 2.15 that $t_{0}^{h}\in K(\bar M)^{0}$. But $K({\bar M})^{0}$ is the stabilizer of the unique global coheir $\bar{q_{0}}$ of $q_{0}$, whereby $t(t_{0}^{h})$ realizes $q_{0}$.  On the other hand, we may assume that $h_{0}$ realizes the global heir $\bar{p_{0}}$ of $p_{0}$ (and that $t,h,t_{0}$ are in $\bar M$).  As the stabilizer of $\bar p_{0}$ is $B({\bar M})^{0}$ which contains $h$ it follows that $hh_{0}$ also realizes $\bar{p_{0}}$.  Putting it together we see that $tp(t(t_{0}^{h})hh_{0}/M) = q_{0}*p_{0}$, as required.

\end{proof}

Note that from Theorem 3.4  and  the discussion in subsection 1.2,  we have identified the universal definable minimal flow of  $SL(2,\Q)$.
Moreover we have shown that $q_{0}*p_{0}$ is almost periodic and idempotent.

\subsection{The Ellis group}
Let ${\cal M}$ denote the minimal $G(M)$-flow $S_{G}(M)*q_{0}*p_{0} = cl({\cal I}*{\cal J})$. The Ellis group attached to the flow $(G(M), S_{G}(M))$ is then the
group $(q_{0}*p_{0}*{\cal M}, *)$ which we aim to describe explicitly. 

Remember that the intersection of  $K(M)$ (i.e. $SL(2,\Z)$) and $B(\Q)$ is $B(\Z)$. 
  
\begin{Lemma}
Let $h$ realize $p_{0}$. Let $t\in SL(2,\Z)$. Then 
\begin{itemize}
\item if $t\in B(\Z)$, then $p_0t=t\tp(h'/M)$, for some $h'\in B(\K)^{0}\cap dcl(h,M)$
\item if $t\notin B(\Z)$, then $p_0t=\tp(t'h'/M)$, where $t'\in SL(2,\V)^{0}\cap dcl(h,M)$ and $h'\in B(\K)\cap dcl(h,M)$.
\end{itemize}
\end{Lemma}
\begin{proof}
The first case is immediate as $p_{0}(x)$ implies $x\in B(\K)^{0}$, $B(\K)^{0}$ is normal in $B(\K)$ and $t\in B(\K)$. 

For the second case: Let $t=\begin{bmatrix} {u_1} & {u_2}\\ {u_3} & {u_4}   \end{bmatrix}$ such that $u_3\neq 0$. Let $h = (a,c)$ realize $p_{0}$.  Then
\[
ht=\begin{bmatrix} {au_1+cu_3} & {au_2+cu_4}\\ {a^{-1}u_3} & {a^{-1}u_4}   \end{bmatrix}=\begin{bmatrix} {1} & {0}\\ {\frac{a^{-1}u_3}{au_1+cu_3}} & {1}   \end{bmatrix}\begin{bmatrix} {au_1+cu_3} & {au_2+cu_4}\\ {0} & {(au_1+cu_3)}^{-1}   \end{bmatrix}=t'h'.
\]
Since $v(c)<dcl(v(a), \mathbb Z)$, we have that  $au_{1} + cu_{3} \neq 0$ and  $v(\frac{a^{-1}u_3}{au_1+cu_3})=v(a^{-1}u_3)-v(au_1+cu_3)=v(a^{-1}u_3)-v(cu_3)>\mathbb Z$. So $st(\frac{a^{-1}u_3}{au_1+cu_3})=0$. This implies that $t'=\begin{bmatrix} {1} & {0}\\ {\frac{a^{-1}u_3}{au_1+cu_3}} & {1}   \end{bmatrix}\in SL(2,\V)^{0}$. Clearly $t'$ and $h'$ are definable over $M,h$.
\end{proof}

\begin{Lemma} $q_{0}*p_{0}*{\cal M} =q_{0}*{\cal J}$.

\end{Lemma} 
\begin{proof} We first prove that $q_{0}*{\cal J} \subseteq  q_{0}*p_{0}*{\cal M}$.

Let $q_{0}*p$ be in the left hand side, namely $p\in {\cal J}$.  Using  Corollary 2.11(iii), we have that $q_{0}*p = q_{0}*p_{0}*p$, and is clearly in ${\cal I}*{\cal J}$, so in $\cal M$. 
\newline
So it suffices to show that $q_{0}*p_{0}*p = q_{0}*p_{0}*q_{0}*p_{0}*p$ which is immediate as $q_{0}*p_{0}$ is an idempotent (Theorem 3.4). 

\vspace{2mm}
\noindent
We now want to show that  $$q_{0}*p_{0}*{\cal M} \subseteq  q_{0}{\cal J}$$

By Lemma 3.3 it suffices to prove that

$$ q_{0}*p_{0}*S_{K}(M)*{\cal J} \subseteq  q_{0}*{\cal J}$$

Let $q\in S_{K}(M)$ and $p\in {\cal J}$.  Let $r\in SL(2,\Z)$ be the standard part of $q$ and let $q' = r^{-1}q$. So $q'\in K^{0}$, and

$$q_{0}*p_{0}*q*p = q_{0}*p_{0}*r*q'*p$$

Now we have two cases:
\newline
{\em Case (i).}  $r\in B(\Z)$.  

Let $t$ realize $q_{0}$, $h$ realize the heir of $p_{0}$ over $(M,t)$, $t'$ realize the heir of $q'$ over $(M, t, h)$ , with $t,h,t'\in G(\bar M)$ and let $h'$ realizes the global heir of $p$. 

By the first part of Lemma 3.5 and our case analysis, $hr = rh_{1}$ with $h_{1}\in B(\K)^{0}\cap dcl(M,h)$. So

\begin{equation*}
\begin{split}
&q_0*p_0*r*q'*p\\
&=\tp(trh_{1}t'h'/M)=\tp(tr{t'}^{h_{1}}h_{1}h'/M)\\
&=\tp(tr{t'}^{h_{1}}/M)*\tp(h_{1}h'/M).
\end{split}
\end{equation*}

But $t'^{h_{1}}$ is in $K^{0}$ (as $h_{1}\in dcl(M,h)$ and we can use Corollary 2.15), and $tr$ realizes the unique coheir over $(M,t'^{h_{1}})$ of the generic type $q_{0}r$ of $K$, whereby $tp(trt'^{h_{1}}/M) = q_{0}r$. 
As before $tp(h_{1}h'/M) = p$.  We have shown so far that $q_{0}*p_{0}*q*p = q_{0}r*p = q_{0}*rp$. As $r$ is assumed to be in $B(\Z)$ we see that $rp\in{\cal J}$ too. So $q_{0}*rp\in q_{0}*{\cal J}$ as required. 

\vspace{2mm}
\noindent
{\em Case (ii).}  $r\in SL(2,\Z)\setminus B(\Z)$.
\newline
By the second part of Lemma 3.5, $p_{0}r$ = $tp(t_{0}h_{0}/M)$ with $t_{0}\in SL(2,\V)^{0}$, $h_{0}\in B(\K)$ and both $t_{0}, h_{0}\in dcl(M,h)$, for $h = (a,c)$ realizing $p_{0}$.
Now choose $t$ realizing the unique coheir of $q_{0}$ over $(M,h)$ and $t'$ realizing the unique heir of $q'$ over $(M,t,h)$, with $t, t', h$ in $G({\bar M})$. Now let $h'$ realize the unique heir of $p$ over ${\bar M}$.   So, by the remarks above,

$$q_{0}*p_{0}*r*q'*p = tp(tt_{0}h_{0}t'h'/M) = tp(tt_{0}(h_{0}t'h_{0}^{-1})h_{0}h'/M).$$

Now, as $t_{0}$ and $t'$ are in $SL(2,\V)^{0}$ and using Corollary 2.15, we see that $t_{0}(h_{0}t'h_{0}^{-1}) \in SL(2,\V)^{0}$, and as $t$ realizes the unique coheir of $q_{0}$ over these elements, $tp(tt_{0}h_{0}t'h_{0}^{-1}/M) = q_{0}$.  On the other hand, now standard arguments give that $tp(h_{0}h'/{\bar M})$ is the unique global heir of $tp(h_{0}h'/M)\in {\cal J}$.  Hence $tp(tt_{0}(h_{0}t'h_{0}^{-1})h_{0}h/M)$ is of the form $q_{0}*p'$ for some $p'\in {\cal J}$,  and Case (ii) is complete.

\end{proof}

We have shown that the Ellis group attached to the flow $(G(M), S_{G}(M))$ is  $q_{0}*{\cal J}$.

\begin{Theorem}  The map from ${\cal J}$ to $q_{0}*{\cal J}$ which takes $p$ to $q_{0}*p$,  is a group isomorphism between $({\cal J},*)$ and $(q_{0}*{\cal J}, *)$.
\end{Theorem}
\begin{proof}  We first show that for $p, p'\in {\cal J}$, $q_{0}*p = q_{0}*p'$ iff $p=p'$. 
\newline
Suppose that $q_{0}*p = q_{0}*p'$. Hence there are realizations $t, t'$ of $q_{0}$, $h$ of $p$ and $h'$ of $p'$ such that  $th = t'h'$.  Note that $t$ and $t'$ are both in $SL(2,\V)^{0}$. So $(t')^{-1}t = h'h^{-1}\in SL(2,\V)^{0}\cap B(\K)$. But $SL(2,\V)^{0}\cap B(\K)$ is easily seen to be $B(\V)^{0}$ which is contained in $B(\K)^{0}$. This shows that $h$ and $h'$ are in the same coset of $B(\K)^{0}$ in $B(\K)$, which implies that $p = p'$. 

So we have shown that the map taking $p\in {\cal J}$ to $q_{0}*p$ establishes a bijection between ${\cal J}$ and $q_{0}*{\cal J}$.  So the Theorem will be established after proving that for $p, p'\in {\cal J}$,
$$ q_{0}*p*q_{0}*p' = q_{0}*p*p'$$

\vspace{2mm}
\noindent
{\em Claim.} Let $p,p'\in {\cal J}$. Then $p*q_{0}*p' = tp(t_{0}/M)*p*p'$ for some $t_{0}\in SL(2,\V)^{0}$.
\newline
{\em Proof of claim.} Let $h_{0}$ realize $p$ in ${\bar M}$, let $t_{1}\in SL(2,\V)^{0}$ realize the unique heir of $q_{0}$ over $M, h_{0}$, and let $a$ realize the unique global heir of 
$p$. Then $p*q_{0}*p = tp(h_{0}t_{1}a/M) = tp((h_{0}t_{1}h_{0}^{-1})h_{0}a/M)$. 
\newline
Put $t_{0} = h_{0}t_{1}h_{0}^{-1}$ which is in $SL(2,\V)^{0}$ (i.e. in $K(\bar M)^{0}$ with earlier notation).  By Lemma 2.10 and Corollary 2.11, $tp(h_{0}a/{\bar M})$ is definable over $M$, and clearly $h_{0}a$ realizes $p*p'$. So  $t_{0}h_{0}a$ realizes $tp(t_{0}/M)*(p*p') = tp(t_{0}/M)*p*p'$, proving the claim. 

\vspace{2mm}
\noindent
Now fix $p, p'\in {\cal J}$. Let $t_{0}$ be given by the claim, and let $h_{0}$, $h_{1}$ be realizations of $p, p'$ respectively in ${\bar M}$ such that  $h_{0}$ realizes the unique heir of $p$ over $M, t_{0}$ and $h_{1}$ realizes the unique heir of $p'$ over $M, t_{0},h_{0}$. So $t_{0}h_{0}h_{1}$ realizes $tp(t_{0}/M)*p*p' = p*q_{0}*p'$ (by the claim).  Let $t$ realize the unique global coheir $q_{0}'$ of $q_{0}$. As $Stab(q_{0}') = SL(2,\V)^{0}$, it follows that $q_{0}'t_{0} = q_{0}'$.  Hence, putting everything together,

$$q_{0}*p*q_{0}*p' = q_{0}*tp(t_{0}/M)*p*p' = tp(tt_{0}/M)*p*p' = q_{0}*p*p'$$

as required.  This completes the proof of Theorem 3.7.

\end{proof}

\begin{Cor} The Ellis group attached to the action of $SL(2,\Q)$ on its type space is $\hat\z$ (as an abstract group).
\end{Cor}
\begin{proof} By 2.11, $({\cal J,*})$ is isomorphic to $B(\K)/B(\K)^{0}$ which is in turn isomorphic to $\hat\z$ by Remark 2.5 and Lemma 2.9.

\end{proof}

\begin{Question} What is the (definable) generalized Bohr compactification of $SL(2,\Q)$?   Does it already coincide with the (definable) Ellis group identified in Theorem 3.7. 
\end{Question}
\noindent
{\em Explanation.}  The generalized Bohr compactification was defined in \cite{Glasner} as a certain quotient of the Ellis group, namely by the intersections of the closures of the neighbourhoods of the identity in the so-called $\tau$-topology on the Ellis group.   This account of the generalized Bohr compactification was discussed in \cite{Krupinski-PillayII} and studied further there  in  the model-theoretic context.

\section{The  action of $SL(2,\Q)$ on the type space of the projective line over $\Q$}

Let $\pp^{1}(\Q)$ denote the projective line over $\Q$, naturally a definable set in $M$. $S_{\pp^{1}}(M)$ denotes the space of complete types over $M$ which concentrate on the definable set $\pp^{1}(\Q)$.  The usual action of $SL(2,\Q)$ on $\pp^{1}(\Q)$ extends to an action on $S_{\pp^{1}}(M)$. We will study this action and observe that the collection of nonalgebraic types in $S_{\pp^{1}}(M)$ is a minimal proximal $SL(2,\Q)$ flow.

\vskip 0.2 cm

We begin with some prequisites concerning projective space and compatibilities with our earlier notation.

 $\pp^1(\Q)$ is defined to be the set of equivalence classes of vectors $\begin{bmatrix} {a_0} \\  {a_1}   \end{bmatrix}$ of elements of $\Q$, not both zero, under the equivalence relation given by $\begin{bmatrix} {a_0} \\  {a_1}   \end{bmatrix}\sim\begin{bmatrix} {\lambda a_0} \\  {\lambda a_1}   \end{bmatrix}$ for all $\lambda \in \Q, \lambda\neq 0$.

$\pp^{1}(\Q)$ is of course interpretable in the structure $M$, and we can definably identify  it with $\Q\cup\{\infty\}$ by identifying the $\sim$-class of  $\begin{bmatrix} {a} \\  {1}   \end{bmatrix}$ with $a\in\Q$ and denoting the $\sim$-class of $\begin{bmatrix} {1} \\  {0}   \end{bmatrix}$ by $\infty$, the point at infinity.  Here $\infty$ is some fixed  tuple from $M$. From now on  we may write  $\begin{bmatrix} {a} \\  {b}   \end{bmatrix}$ instead of its $\sim$-class. 

 Note that  $\pp^{1}(\Q)$ is a $p$-adic analytic manifold via the natural bijections $\phi_1: \Q\longrightarrow\pp(\Q)\setminus\left\{\begin{bmatrix} {0} \\  {1}   \end{bmatrix}\right\}$ and $\phi_2: \Q\longrightarrow\pp(\Q)\setminus\left\{\begin{bmatrix} {1} \\  {0}   \end{bmatrix}\right\}$ which give $\pp(\Q)$ a manifold structure. This $p$-adic manifold structure is also definable in the structure $M$. 

But  we will be mainly interested in   $\pp^{1}(\Q)$ as a definable set in $M$. 
$\pp^{1}(\K)$ denotes the obvious thing, and in fact we can consider $\pp^{1}$ as a formula in the language of $M$. 

The standard action of $G(\Q)$ on $\pp^{1}(\Q)$ is: $\begin{bmatrix} {a}&b \\  {c} &d  \end{bmatrix}\cdot\begin{bmatrix} {x} \\  {y}   \end{bmatrix}=\begin{bmatrix} {ax+by} \\  {cx+dy}   \end{bmatrix}$; this action is well-defined since $\begin{bmatrix} {a}&b \\  {c} &d  \end{bmatrix}$ is invertible. 
Moreover the same formula gives an action of $G(\K)$ on $\pp^{1}(\K)$. 
In any case we obtain an action of $G(\Q)$ on the compact space $S_{\pp^{1}}(M)$ which is a definable action as discussed earlier.

\begin{Rmk}
\begin{itemize}
\item The stabilizer of $\begin{bmatrix} {1} \\  {0}   \end{bmatrix}$ is $B(\Q)$.
\item The quotient space $G(\Q)/B(\Q)$ is homeomorphic to $\pp(\Q)$ via:
\[
\begin{bmatrix} {a}&b \\  {c} &d  \end{bmatrix}/B(\Q)\mapsto\left\{
\begin{array}{rcl}
\begin{bmatrix} {a/c} \\  {1}   \end{bmatrix}       &      & {\text{if}\ c\neq 0}\\
&  &\\
\begin{bmatrix} {1} \\  {0}   \end{bmatrix}       &      & {\text{if}\ c= 0}
\end{array} .\right.
\]
\end{itemize}
\end{Rmk}


\vskip 0.2 cm

\begin{Rmk} We have given above a  definable identification of $\pp^{1}(\Q)$ with $\Q \cup\{\infty\}$.  The same thing identifies $\pp^{1}(\K)$ with $\K \cup\{\infty\}$. Hence the  type space $S_{\pp^{1}}(M)$ identifies with the space $S_{1}(M)$ of complete $1$-types over $M$, together with the point $\infty$, which is considered as a realized type.  Note that with notation from Section 2.1, the $1$-types over $M$ of the form $p_{\infty, C}$ will be the types of elements of $\pp^{1}(\K)$ which are infinitesmially close to the point $\infty$ (with respect to the $p$-adic manifold topology discussed earlier).  
\end{Rmk}

\begin{Def}  For $p\in S_{G}(M)$ and $q$ in $S_{\pp^{1}}(M)$ we define $p*q$ as $tp(g\cdot b/M)$ where $b$ realizes $q$ and $g$ realizes the unique coheir of $p$ over $(M,b)$.
\end{Def}

\begin{Rmk} (i) If $p_{1}, p_{2}\in S_{G}(M)$ and $q\in S_{\pp^{1}}(M)$ then $(p_{1}*p_{2})*q = p_{1}*(p_{2}*q)$.
\newline
(ii) Let  $\pi$ be the map from $SL(2,\Q)$ onto $\pp^{1}(\Q)$ defined implicitly in 4.1, extended naturally to a map between the respective type spaces.
Then for $p_{1}, p_{2}\in S_{G}(M)$, $p_{1}*\pi(p_{2}) = \pi(p_{1}*p_{2})$.
\newline
(iii)  For any $q\in S_{\pp^{1}}(M)$ the closure of the $G(M)$-orbit $G(M)\cdot q$ is precisely $\{p*q:p\in S_{G}(M)\}$. 
\end{Rmk}


\vskip 0.4 cm

We now use some notation from earlier sections. Specifically $p_{0}\in S_{B}(M)\subset  S_{G}(M)$ and $q_{0}\in S_{K}(M)\subset S_{G}(M)$ are specific $f$-generic types. 
Let
\
\\
$TP_\infty=\left\{\tp\left(\begin{bmatrix} {a} \\  {1}   \end{bmatrix} /\Q\right)\in S_{\pp^{1}}(\Q)|\ v(a)<n:n\in \mathbb Z\right\}$.

So $TP_{\infty}$ is the infinitesimal neighbourhood of $\infty$ in $S_{\pp^{1}}(M)$ (with the topology coming from the manifold topology on $\pp^{1}(\Q)$), minus the point $\infty$ itself. 

Then 
\begin{Lemma}\label{pi_{p0}(q)}
For every $q\in S_{\pp^{1}}(\Q)$, 
\begin{itemize}
\item if $q\neq \infty$, then $p_{0}*q\in TP_\infty$.
\item $p_{0}*\infty = \infty$.
\end{itemize}
\end{Lemma}
\begin{proof}
\begin{itemize}
 \item Suppose that $q\in S_\pp^{1}(\Q)$ and  $q\neq \infty$. Let $\begin{bmatrix} {a} \\  {1}   \end{bmatrix}$ be a realization of $q$ and $h=\begin{bmatrix} {b} & c\\  {0} &b^{-1}  \end{bmatrix}$  a  realization of $p_0$ such that $\tp(a/\Q,b,c)$ is the heir of $\tp(a/\Q)$.
Then \[p_{0}*q =\tp\left(\begin{bmatrix} {b} & c\\  {0} &b^{-1}  \end{bmatrix}\cdot \begin{bmatrix} {a} \\  {1}   \end{bmatrix}/\Q\right)=\tp(ab^2+bc/\Q).\]
If $v(a)>n$ for some $n\in\mathbb Z$, since $v(c)<dcl(\mathbb Z, b)$ we have $v(bc)<v(ab^2)$ and thus $v(ab^2+bc)=v(bc)<\mathbb Z$; if $v(a)<\mathbb Z$, then $v(ab^2)<v(bc)$ since $\tp(v(a)/\mathbb Z, v(b),v(c))$ is an heir of $\tp((v(a)/\mathbb Z)$, so $v(ab^2+bc)=v(ab^2)<\mathbb Z$. This implies that $p_{0}*q\in TP_\infty$.

 \item As $B(\K)$ stabilizes $\infty$ and $p_0\in S_B(\Q)$, we have that  $p_{0}*q=q$ where $q = \infty$.
 \end{itemize}
\end{proof}

\begin{Lemma}\label{pi_q0(q)}
For every $q\in TP_\infty$, we have $q_{0}*q=q_{0}*\infty$.
\end{Lemma}
\begin{proof}
Suppose that $q\in TP_\infty$ and $q\neq \infty$. Let $\begin{bmatrix} {c} \\  {1}   \end{bmatrix}$ be a realization of $q$. Then $\begin{bmatrix} {c} \\  {1}   \end{bmatrix}=\begin{bmatrix} {1} &0\\ c^{-1} &{1}   \end{bmatrix}\cdot\begin{bmatrix} {1} \\  {0}   \end{bmatrix}$. Since $v(c)<\mathbb Z$, $st\bigg(\begin{bmatrix} {1} &0\\ c^{-1} &{1}   \end{bmatrix}\bigg)=\begin{bmatrix} {1} &0\\ 0 &{1}   \end{bmatrix}$. So $\begin{bmatrix} {1} &0\\ c^{-1} &{1}   \end{bmatrix}\in K^0$. Let $p=\tp(\begin{bmatrix} {1} &0\\ c^{-1} &{1}   \end{bmatrix}/\Q)$. Then $q=p*(\begin{bmatrix} {1} \\  {0}   \end{bmatrix})$. So
\[
q_{0}*q =q_{0}*(p*(\begin{bmatrix} {1} \\  {0}   \end{bmatrix}))= (q_0*p)*(\begin{bmatrix} {1} \\  {0}   \end{bmatrix}).
\]
Since $q_0$ is generic in $S_{K}(M)$ and $p$ is realized by some element from $K^0$, we have $q_0*p=q_0$.
\end{proof}
By Lemma \ref{pi_{p0}(q)} and Lemma \ref{pi_q0(q)}, we have

\begin{Theorem}
$(q_0*p_{0})*(S_{\pp^{1}}(\Q))=\left\{q_{0}*\left(\begin{bmatrix} {1} \\  {0}   \end{bmatrix} \right)\right\}$.
\end{Theorem}

\begin{Cor} The set of nonalgebraic types in $S_{\pp^{1}}(M)$ is a minimal proximal $SL(2,\Q)$-flow.
\end{Cor}
\begin{proof} Firstly the set $S_{\pp^{1},na}(M)$ of nonalgebraic types in $S_{\pp^{1}}(M)$ is closed and $SL(2,\Q)$-invariant.  Secondly (see Remark 4.4(iii)), any minimal $SL(2,\Q)$-subflow of $S_{\pp^{1},na}(M)$, is closed under $S_{G}(M)*$, so by Theorem 4.7 contains 
$q_{0}*\left(\begin{bmatrix} {1} \\  {0}   \end{bmatrix} \right)$. We have shown so far that  $S_{\pp^{1}}(M)$  has a unique minimal subflow which is the closure of the orbit of $q_{0}*\infty$ and that this minimal subflow is proximal.  It remains to see that $S_{\pp^{1},na}(M)$ is a minimal subflow. It is clearly closed and $SL(2,\Q)$-invariant.  Note that the unique minimal subflow of $S_{\pp^{1}}(M)$ contains a minimal subflow with respect to the multiplicative group, which by Proposition 2.4 consists either of the complete nonalgebraic $1$-types over $M$ which are ``infinitesimally close" to $0$, or the complete nonalgebraic $1$-types over $M$  which are ``at infinity", namely what we called above the infinitesimal neighbourhood of $\infty$ in $S_{\pp^{1},na}(M)$.   As $SL(2,\Q)$ acts transitively on $\pp^{1}(\Q)$, it follows that for every $a\in \pp^{1}(\Q)$, the set of complete nonalgebraic  $1$-types over $M$ infinitesimally close to $a$ is included in the unique minimal subflow of $S_{\pp^{1},na}(M)$. But this accounts for all of $S_{\pp^{1},na}(M)$ which is therefore minimal is claimed. 
\end{proof}

\begin{Question}
(i) Is  $S_{\pp^{1},na}(M)$ the universal minimal proximal definable $SL(2,\Q)$ flow?
\newline
(ii) Is  $S_{\pp^{1},na}(M)$  a strongly proximal $SL(2,\Q)$-flow?
\end{Question}
\noindent
{\em Explanation.}  (i) The universal minimal proximal definable flow exists and will be a minimal proximal $SL(2,\Q)$-flow which is a ``homomorphic image" of the universal minimal definable flow, and universal such. 
\newline
(ii)  Strong proximality of $(X,G)$ means that the action of $G$ on the space of Borel probability measures on $X$ is proximal.  In our context, $(S_{\pp^{1},na}(M), SL(2,\Q))$, the action  will be definable. See Proposition 6.3 of \cite{Krupinski-PillayII}.  We guess that the answer to (ii) is positive.


\begin{thebibliography}{99}


\bibitem{Auslander} J.Auslander, {\em Minimal flows and their extensions}, North Holland, Amsterdam, 1988.

\bibitem{Belair} Luc Belair, Panorama of $p$-adic model theory, Ann. Sci. Math. Quebec, 36(1) · January 2012

\bibitem{Bruhat-Tits} F. Bruhat and J. Tits, Groupes r\'{e}ductive sur un corps local, Publ. Math. IHES (1972), vol. 41, 5-251.


\bibitem{C-P-S} A. Chernikov, A. Pillay, P. Simon, External definability and groups in NIP theories, J. London Math. Soc. vol. 90 (2014) 213–240

\bibitem{C-S} A. Chernikov and P. Simon, Definably amenable $NIP$ groups, J. American Math. Society, 31 (2018), 609-641. 

\bibitem{Del89} Fran\c coise Delon. D\'efinissabilit\'e avec param\`etres ext\'eriours dans $\Q$ et $\R$. Proceedings of the American Mathematical Society, 106(1): 193-198, 1989.  

\bibitem{Druart-thesis} B. Druart, Ph.D. thesis, University Lyon 1, 2015. 

\bibitem{Druart-paper} B. Druart, Definable subgroups in $SL_{2}$ over a $p$-adically closed field, arXiv: 1501.06834v1, 2015. 


\bibitem{Ellis}  R. Ellis, {\em Lectures on topological dynamics}, Benjamin, 1969. 

\bibitem{GPPI} J. Gismatullin, D. Penazzi and  A. Pillay, Some model theory of SL($2,\R$), Fundamenta Mathematicae 229(2), (2015).

\bibitem{GPPII} J. Gismatullin, D. Penazzi and  A. Pillay, On compactifications and the topological dynamics of definable groups, Annals of  Pure and Applied Logic, 165(2014), 552-562.

\bibitem{Glasner} E. Glasner, {\em Proximal Flows}, Springer Lecture Notes 517,  Springer, 1976.

\bibitem{NIPI} E. Hrushovski, Y. Peterzil, and A. Pillay, Groups, measures, and the NIP, Journal AMS 21 (2008), 563-596.

\bibitem{NIPII} E. Hrushovski and A. Pillay, On NIP and invariant measures, J. European Math. Soc. 13 (2011), 1005 - 1061.






\bibitem{G. Jagiella} G. Jagiella, Definable topological dynamics and real Lie groups, Math. Logic Quarterly, 61 (2015), 45 - 55.

\bibitem{Krupinski-PillayII} K. Krupinski and A. Pillay, Generalised Bohr compactification and model-theoretic connected components, Math Proceedings of Cambridge Philosophical Society, Volume 163, Issue 2 September 2017 , pp. 219 - 249.

\bibitem{Krupinski-Pillay} K. Krupinski and A. Pillay, Amenability, definable groups, and automorphism groups, Advances in Math. vol. 345 (2019), 1253-1299. 

\bibitem{Macintyre} A. Macintyre, On definable subsets of $p$-adic fields, Journal of Symbolic Logic, 41 (1976), 605-610.

\bibitem{Newelski} L. Newelski, Topological dynamics of definable group actions, Journal of Symbolic Logic, 74(2009), 50-72.

\bibitem{O-P} A. Onshuus and A. Pillay,  Definable Groups and Compact $p$-adic Lie Groups, Journal of the London Mathematical Society, 78(1), (2008) , 233-247.
\bibitem{Pillay-intro-stability} A. Pillay, {\em An introduction to stability theory}, Oxford University Press, 1983. 
\bibitem{Pillay-top}  A. Pillay, Topological dynamics and definable groups, J. Symbolic Logic, 78:657-666,2013.
\bibitem{P-Y} A. Pillay and N. Yao, On minimal flows, definably amenable groups, and o-minimality, Adv. in Mathematics, 290(2016), 483-502.

\bibitem{Poizat-book} B. Poizat, {\em A course in model theory}, Spinger-Verlag, NY, 2000. 

\bibitem{Simon}  P. Simon, VC-sets and generic compact domination, Israel J. Math., 218 (2017), 27-41.

\bibitem{Y-L} N. Yao and D. Long, Topological dynamics for groups definable in real closed fields, Annals of Pure and Applied Logic, 166(2015), 261-273.













\end{thebibliography}
\end{document}